\documentclass[a4paper,12pt]{amsart}

\usepackage{a4wide}

\usepackage{amsmath}
\usepackage{amssymb}
\usepackage{amsthm}

\usepackage{hyperref}

\newtheorem{theorem}{Theorem}
\newtheorem{lemma}[theorem]{Lemma}
\newtheorem{definition}[theorem]{Definition}

\usepackage[foot]{amsaddr}

\newcommand{\vv}{\boldsymbol v}
\newcommand{\tvv}{\widetilde{\vv}}
\newcommand{\hvv}{\widehat{\vv}}
\newcommand{\ww}{\boldsymbol w}

\newcommand{\rr}{\boldsymbol r}
\newcommand{\trr}{\widetilde \rr}
\newcommand{\hrr}{\widehat   \rr}
\newcommand{\xxi}{\boldsymbol \xi}

\newcommand{\xx}{\boldsymbol x}
\newcommand{\yy}{\boldsymbol y}

\newcommand{\XX}{\boldsymbol X}

\newcommand{\RRR}{\boldsymbol R}

\newcommand{\nn}{\boldsymbol n}
\newcommand{\ttau}{{\boldsymbol \tau}}

\newcommand{\LL}{\boldsymbol L}
\newcommand{\HH}{\boldsymbol H}

\newcommand{\CC}{\pmb {\mathcal C}}

\newcommand{\ccurl}{\boldsymbol{\operatorname{curl}}}
\newcommand{\scurl}{\operatorname{curl}}

\newcommand{\ddiv}{\operatorname{div}}

\newcommand{\grad}{\boldsymbol \nabla}
\renewcommand{\div}{\grad \cdot}
\newcommand{\curl}{\grad \times}

\newcommand{\trc}{\operatorname{trc}}

\newcommand{\VV}{\boldsymbol V}

\newcommand{\FF}{\mathcal F}

\newcommand{\PP}{\pmb{\mathcal P}}
\newcommand{\NN}{\pmb{\mathcal N}}
\newcommand{\RT}{\pmb{\mathcal{RT}}}

\newcommand{\ppsi}{\boldsymbol \psi}
\newcommand{\pphi}{\boldsymbol \phi}

\newcommand{\ttheta}{\boldsymbol \theta}

\newcommand{\edge}{e}

\newcommand{\ppi}{\boldsymbol \pi}

\newcommand{\hK}{{\widehat K}}
\newcommand{\hF}{{\widehat F}}
\newcommand{\hFF}{{\widehat{\FF}}}

\newcommand{\TTT}{\boldsymbol T}
\newcommand{\JJJ}{\mathbb J}

\newcommand{\Kin} {\hK}
\newcommand{\Kout}{K}

\newcommand{\eq}{:=}

\newcommand\eg{e.g.}

\newcommand{\bse}{\begin{subequations}}
\newcommand{\ese}{\end{subequations}}

\hypersetup{colorlinks, linkcolor=blue, citecolor=blue,
urlcolor=blue, plainpages=false, pdfwindowui=false,
pdfstartview={FitH}}

\title[p-robust H(curl)-stability of discrete minimization]{Polynomial-degree-robust $\HH(\ccurl)$-stability of discrete minimization in a tetrahedron$^\star$}

\author{T. Chaumont-Frelet$^{1,2}$}
\author{A. Ern$^{3,4}$}
\author{M. Vohral\'ik$^{4,3}$}
\address{\vspace{-.5cm}}
\address{\noindent \tiny \textup{$^\star$This project has received funding from the European Research Council (ERC) under the European Union’s Horizon 2020}}
\address{\noindent \tiny \textup{\hspace{.2cm}research and innovation program (grant agreement No 647134 GATIPOR).}}
\address{\noindent \tiny \textup{$^1$Inria, 2004 Route des Lucioles, 06902 Valbonne, France}}
\address{\noindent \tiny \textup{$^2$Laboratoire J.A. Dieudonn\'e, Parc Valrose, 28 Avenue Valrose, 06108 Nice Cedex 02, 06000 Nice, France}}
\address{\noindent \tiny \textup{$^3$Universit\'e Paris-Est, CERMICS (ENPC), 6 et 8 av. Blaise Pascal 77455 Marne la Vall\'ee cedex 2, France}}
\address{\noindent \tiny \textup{$^4$Inria, 2 rue Simone Iff, 75589 Paris, France}}

\begin{document}

\maketitle

\begin{abstract}
We prove that the minimizer in the N\'ed\'elec polynomial space of some degree $p\ge0$
of a discrete minimization problem performs as well as the continuous minimizer
in $\HH(\ccurl)$, up to a constant that is independent of the polynomial degree $p$.
The minimization problems are posed for fields defined on a single non-degenerate tetrahedron
in $\mathbb R^3$ with polynomial constraints enforced on the curl of the field and its tangential
trace on some faces of the tetrahedron. This result builds upon [L.~Demkowicz, J.~Gopalakrishnan, J.~Sch{\"o}berl, \emph{SIAM J. Numer. Anal.} \textbf{47} (2009), 3293--3324] and [M.~Costabel, A.~McIntosh, \emph{Math. Z.} \textbf{265} (2010), 297--320] and is a fundamental ingredient to build polynomial-degree-robust a posteriori error estimators when approximating the Maxwell equations in several regimes leading to a curl-curl problem.


\end{abstract}

\section{Introduction}

When discretizing the Poisson equation with Lagrange finite elements,
flux equilibrated error estimators can be employed to build polynomial-degree-robust
(or $p$-robust for short) a posteriori error estimators
\cite{Brae_Pill_Sch_p_rob_09,Ern_Voh_p_rob_15}. This property, which is particularly important
for $hp$-adaptivity (see for instance~\cite{Dan_Ern_Sme_Voh_guar_red_18} and the references therein), means that
the local a posteriori error estimator is, up to data oscillation, a lower bound
of the local approximation error, up to a constant that is independent of the
polynomial degree (the constant can depend on the shape-regularity of the mesh).
It turns out that one of the cornerstones of $p$-robust local efficiency
is a $p$-robust $\HH(\ddiv)$-stability result
of a discrete minimization problem posed in a single mesh tetrahedron.
More precisely, let $K \subset \mathbb R^3$ be a non-degenerate tetrahedron and let
$\emptyset\subseteq \FF \subseteq \FF_K$ be a (sub)set of its faces.
Then there is a constant $C$ such that for every polynomial degree $p\ge0$ and all polynomial data
$r_K \in \mathcal P_p(K)$ and $r_F \in \mathcal P_p(F)$ for all
$F\in \FF$, such that $(r_K,1)_K=\sum_{F\in\FF} (r_F,1)_F$ if $\FF=\FF_K$
(detailed notation is explained below), one has
\begin{equation}
\label{stable_minimization_Hdiv}
\min_{\substack{
\vv_p \in \RT_p(K)
\\
\div \vv_p = r_K
\\
\vv_p \cdot \nn_K|_F = r_F \; \forall F \in \FF
}}
\|\vv_p\|_{0,K}
\leq
C
\min_{\substack{
\vv \in \HH(\ddiv,K)
\\
\div \vv = r_K
\\
\vv \cdot \nn_K|_F = r_F \; \forall F \in \FF
}}
\|\vv\|_{0,K}.
\end{equation}
This result is shown in~\cite[Lemma~A.3]{Ern_Voh_p_rob_3D_20},
and its proof relies on~\cite[Theorem~7.1]{Demk_Gop_Sch_ext_III_12} and~\cite[Proposition~4.2]{Cost_McInt_Bog_Poinc_10}. Importantly, the constant $C$  in~\eqref{stable_minimization_Hdiv} only depends on the shape-regularity of $K$, that is,
the ratio of its diameter to the diameter of its largest inscribed ball.
Notice that the converse bound of~\eqref{stable_minimization_Hdiv} trivially
holds with constant $1$. The stability result stated in~\eqref{stable_minimization_Hdiv}
is remarkable since it states that the minimizer from the discrete minimization set
performs as well as the minimizer from the continuous minimization set, up to a
$p$-robust constant.

The main contribution of the present work is to establish the counterpart of
\eqref{stable_minimization_Hdiv} for the N\'ed\'elec finite elements of order $p\ge0$ and
the Sobolev space $\HH(\ccurl)$.
As in the $\HH(\ddiv)$ case, our discrete stability result
relies on two key technical tools: a
stable polynomial-preserving lifting of volume data from
\cite[Proposition~4.2]{Cost_McInt_Bog_Poinc_10}, and
stable polynomial-preserving liftings of boundary data from
\cite{Demk_Gop_Sch_ext_I_09,Demk_Gop_Sch_ext_II_09,Demk_Gop_Sch_ext_III_12}.
Our main result, Theorem~\ref{theorem_stability_tetrahedra} below,
may appear as a somewhat expected consequence of these lifting operators, but our
motivation here is to provide all the mathematical
details of the proofs, which turn out to be nontrivial and in particular
more complex than in~\cite[Lemma~A.3]{Ern_Voh_p_rob_3D_20}. In particular 
the notion of tangential traces in $\HH(\ccurl)$ is somewhat delicate, and 
we employ a slightly different definition compared to
\cite{Demk_Gop_Sch_ext_I_09,Demk_Gop_Sch_ext_II_09,Demk_Gop_Sch_ext_III_12}.
Theorem~\ref{theorem_stability_tetrahedra} is to be used as a building block
in the construction of a $p$-robust a posteriori error estimator for
curl-curl problems. This construction will be analyzed in a forthcoming work.

The remainder of this paper is organized as follows.
We introduce basic notions in Section~\ref{sec:preliminaries_K} so as to state our main result, Theorem~\ref{theorem_stability_tetrahedra}. Then
Section~\ref{sec:proof} presents its proof.

\section{Statement of the main result}
\label{sec:preliminaries_K}

\subsection{Tetrahedron}

Let $K \subset \mathbb R^3$ be an arbitrary tetrahedron. We assume that $K$
is non-degenerate, i.e., the volume of $K$ is positive. We employ the notation
\begin{equation*}
h_K \eq \max_{\xx,\yy \in \overline{K}} |\xx-\yy|,
\qquad
\rho_K
\eq
\max \left \{ d \geq 0
\; \left | \;
\exists \xx \in K; \; B\left (\xx,\frac{d}{2} \right ) \subset \overline{K}
\right . \right \},
\end{equation*}
for the diameter of $K$ and the diameter of the largest closed ball contained
in $\overline{K}$. Then $\kappa_K \eq h_K / \rho_K$ is the so-called
shape-regularity parameter of $K$. Let $\FF_K$ be the set of faces of $K$,
and for every face $F \in \FF_K$, we denote by $\nn_F$ the unit vector normal to $F$ pointing outward $K$.

\subsection{Lebesgue and Sobolev spaces}

The space of square-integrable scalar-valued (resp. vector-valued)
functions on $K$ is denoted by $L^2(K)$ (resp. $\LL^2(K)$),
and we use the notation $(\cdot,\cdot)_K$ and $\|\cdot\|_{0,K}$ for, respectively, the inner product and the associated norm of both $L^2(K)$ and $\LL^2(K)$. $H^1(K)$ is the usual Sobolev space of scalar-valued functions
with weak gradient in $\LL^2(K)$, and $\HH^1(K)$ is the space of vector-valued functions
having all their components in $H^1(K)$.

If $F \in \FF_K$ is a face of $K$, then $\LL^2(F)$ is the set of vector-valued functions
that are square-integrable with respect to the surfacic measure of $F$.
For all $\ww \in \HH^1(K)$, we define the tangential component of $\ww$ on $F$ as
\begin{equation}
\label{eq:def_pi_tau_F}
\ppi^\ttau_F(\ww) \eq \ww|_F - (\ww|_F \cdot \nn_F) \nn_F \in \LL^2(F).
\end{equation}
More generally, if $\FF \subseteq \FF_K$ is a nonempty (sub)set of the faces of $K$,
we employ the notation $\Gamma_\FF \subseteq \partial K$ for the corresponding part of
the boundary of $K$, and $\LL^2(\Gamma_{\FF})$ is the associate Lebesgue space of
square-integrable functions over $\Gamma_{\FF}$.

\subsection{N\'ed\'elec and Raviart--Thomas polynomial spaces}

For any polynomial degree $p \geq 0$, the notation $\PP_p(K)$ stands for the space
of vector-valued polynomials such that all their components belong to $\mathcal P_p(K)$
which is composed of the restriction to $K$ of real-valued polynomials of total
degree at most $p$.
Following~\cite{Ned_mix_R_3_80,Ra_Tho_MFE_77}, we define the
polynomial spaces of N\'ed\'elec and Raviart--Thomas functions as follows:
\begin{equation*}
\NN_p(K) \eq \PP_p(K) + \xx \times \PP_p(K)
\quad \text{ and } \quad
\RT_p(K) \eq \PP_p(K) + \xx \mathcal P_p(K).
\end{equation*}

Let $\FF \subseteq \FF_K$ be a nonempty (sub)set of the faces of $K$.
On $\Gamma_\FF$, we define the (piecewise) polynomial space composed
of the tangential traces of the N\'ed\'elec polynomials
\begin{equation}
\label{eq_tr_K}
\NN_p^\ttau(\Gamma_\FF) \eq
\left \{
\ww_\FF \in \LL^2(\Gamma_\FF) \; | \;
\exists \vv_p \in \NN_p(K);
\ww_F \eq (\ww_\FF)|_F = \ppi^\ttau_F (\vv_p) \quad \forall F \in \FF
\right \}.
\end{equation}
Note that $\ww_\FF \in \NN_p^\ttau(\Gamma_\FF)$ if and only if
$\ww_F\in \NN_p^\ttau(\Gamma_{\{F\}})$ for all $F\in\FF$ and whenever $\FF$ contains two or more faces, $|\FF|\ge2$,
for every pair $(F_-,F_+)$ of distinct faces in $\FF$, the
compatibility condition $(\ww_{F_+})|_\edge \cdot \ttau_\edge =
(\ww_{F_-})|_\edge \cdot \ttau_\edge$ holds true on their common 
edge $\edge \eq F_+\cap F_-$, i.e., the tangential trace is continuous along $\edge$.
For all $\ww_\FF \in \NN_p^\ttau(\Gamma_\FF)$, we define its surface curl as
\begin{equation}
\label{eq_scurl_curl_el}
    \scurl_F (\ww_F) \eq (\curl \vv_p)|_F \cdot \nn_F \qquad \forall F \in \FF,
\end{equation}
where $\vv_p$ is any element of $\NN_p(K)$ such that $\ww_F = \ppi_F^{\ttau}(\vv_p)$
for all $F \in \FF$. This function is well-defined independently of the choice of $\vv_p$.

\subsection{Weak tangential traces for fields in $\HH(\ccurl,K)$ by integration by parts}

Let $\HH(\ccurl,K) \eq \left \{
\vv \in \LL^2(K) \; | \; \curl \vv \in \LL^2(K)
\right \}$ denote the Sobolev space composed of square-integrable vector-valued fields
with square-integrable curl. We equip this space with the norm
$\|\vv\|_{\ccurl,K}^2 \eq \|\vv\|_{0,K}^2 + \ell_K^2\|\curl \vv\|_{0,K}^2$,
where $\ell_K$ is a length scale associated with $K$, e.g., $\ell_K\eq h_K$
(the choice of $\ell_K$ is irrelevant in what follows).

For any field $\vv \in \HH^1(K)$, its tangential trace on a face $F\in\FF_K$
can be defined by using~\eqref{eq:def_pi_tau_F}.
This notion of (tangential) trace is defined (almost everywhere)
on $F$ without invoking test functions. The situation for a field 
in $\HH(\ccurl,K)$ is more delicate.
The tangential trace over the whole boundary of $K$
can be defined by duality, but it is not straightforward to define 
the tangential trace on a part of the boundary of $K$.
While it is possible to use restriction operators
\cite{Demk_Gop_Sch_ext_I_09,Demk_Gop_Sch_ext_II_09,Demk_Gop_Sch_ext_III_12},
we prefer a somewhat more direct
definition based on integration by parts. This approach is also more convenient
when manipulating (curl-preserving) covariant Piola transformations
(see, \eg, \cite[Section~7.2]{Ern_Guermond_FEs_I_2020} and Section~\ref{sec_st_2} below),
which is of importance, \eg, when mapping tetrahedra of a mesh to a reference tetrahedron.

In this work, we consider the following definition of the tangential trace on a (sub)set
$\Gamma_\FF \subseteq \partial K$.

\begin{definition}[Tangential trace by integration by parts]
\label{definition_partial_trace}
Let $K \subset \mathbb R^3$ be a non-degenerate tetrahedron and let
$\FF \subseteq \FF_K$ be a nonempty (sub)set of its faces.
Let $\rr_\FF \in \NN_p^\ttau(\Gamma_\FF)$ as well as $\vv \in \HH(\ccurl,K)$.
We will employ the notation ``$\vv|^\ttau_\FF = \rr_\FF$'' to say that
\begin{equation*}
(\curl \vv,\pphi)_K - (\vv,\curl \pphi)_K = \sum_{F \in \FF} (\rr_F,\pphi \times \nn_F)_{F}
\quad
\forall \pphi \in \HH^1_{\ttau,\FF^{\mathrm{c}}}(K),
\end{equation*}
where
\begin{equation*}
\HH_{\ttau,\FF^{\mathrm{c}}}^1(K) \eq
\left \{
\ww \in \HH^1(K) \; | \; \ppi^\ttau_F(\ww) = \boldsymbol 0
\quad
\forall F \in \FF^{\mathrm{c}} \eq \FF_K \setminus \FF
\right \}.
\end{equation*}
Whenever $\vv \in \HH^1(K)$, $\vv|^\ttau_\FF = \rr_\FF$ if and only if
$\ppi^\ttau_F (\vv) = \rr_F$ for all $F \in \FF$.
\end{definition}

\subsection{Main result}
\label{sec_main_result}

We are now ready to state our main result. The proof is given in Section~\ref{sec:proof}.

\begin{theorem}[Stability of $\HH(\ccurl)$ discrete minimization in a tetrahedron]
\label{theorem_stability_tetrahedra}
Let $K \subset \mathbb R^3$ be a non-degenerate tetrahedron and let
$\emptyset\subseteq \FF \subseteq \FF_K$ be a (sub)set of its faces.
Then, for every polynomial degree $p \geq 0$, for all $\rr_K \in \RT_p(K)$
such that $\div \rr_K = 0$, and, if $\emptyset\ne\FF$, for all
$\rr_\FF \in \NN_p^\ttau(\Gamma_\FF)$ such that
$\rr_K \cdot \nn_{F} = \scurl_F (\rr_F)$ for all $F \in \FF$, the following holds:
\begin{equation}
\label{eq_minimization_element_K}
\min_{\substack{
\vv_p \in \NN_p(K)
\\
\curl \vv_p = \rr_K
\\
\vv_{p}|^\ttau_\FF = \rr_\FF
}}
\|\vv_p\|_{0,K}
\le C_{\mathrm{st},K}
\min_{\substack{
\vv \in \HH(\ccurl,K)
\\
\curl \vv = \rr_K
\\
\vv|^\ttau_\FF = \rr_\FF
}}
\|\vv\|_{0,K},
\end{equation}
where the condition on the tangential trace in the minimizing sets is null if $\emptyset=\FF$.
Both minimizers in~\eqref{eq_minimization_element_K} are uniquely defined and the constant
$C_{\mathrm{st},K}$ only depends on the shape-regularity parameter $\kappa_K$ of $K$,
so that it is in particular independent of $p$.
\end{theorem}

\section{Proof of the main result}
\label{sec:proof}

The discrete minimization set in~\eqref{eq_minimization_element_K}, which is a subset of
the continuous minimization set, is nonempty owing to classical properties of
the N\'ed\'elec polynomials and the compatibility conditions imposed on the data
$\rr_K$ and $\rr_\FF$. This implies the existence and uniqueness of both
minimizers owing to standard convexity arguments.


The proof of the bound~\eqref{eq_minimization_element_K} proceeds in three steps.
Fist we establish in Section~\ref{sec_st_1} the bound for minimization problems without trace constraints.
This first stability result crucially relies on~\cite{Cost_McInt_Bog_Poinc_10}
and is established directly on the given tetrahedron $K \subset \mathbb R^3$.
Then we establish in Section~\ref{sec_st_2} the bound for minimization problems without curl constraints.
This second stability result crucially relies on the results of~\cite{Demk_Gop_Sch_ext_II_09,Demk_Gop_Sch_ext_III_12}. Since the notion of
tangential trace employed therein slightly differs from the present one,
we first establish in Section~\ref{sec_aux} some auxiliary results on tangential traces and then prove the stability result by first working on the reference tetrahedron in $\mathbb R^3$
and then by mapping the fields defined on the given tetrahedron $K \subset \mathbb R^3$
to fields defined on the reference tetrahedron.
In all cases, the existence and uniqueness of the minimizers follows by the same arguments
as above.
Finally, in Section~\ref{sec_st_3} we combine both results so as to prove Theorem~\ref{theorem_stability_tetrahedra}.

To simplify the notation we write
$A\lesssim B$ for two nonnegative numbers $A$ and $B$
if there exists a constant $C$ that only depends on the
shape-regularity parameter $\kappa_K$ of $K$ but is independent of $p$
such that $A \leq C B$. The value of $C$ can change at each occurrence.

\subsection{Step 1: Minimization without trace constraints} \label{sec_st_1}

\begin{lemma}[Minimization without trace constraint]
\label{lemma_lifting_curl}
Let $K \subset \mathbb R^3$ be a non-degenerate tetrahedron.
Let $\rr_{K} \in \RT_p(K)$ be such that $\div \rr_{K} = 0$. The following holds:
\begin{equation}
\label{eq_minimization_zero_face}
\min_{\substack{
\vv_p \in \NN_p(K)
\\
\curl \vv_p = \rr_{K}
}}
\|\vv_p\|_{0,K}
\lesssim
\min_{\substack{
\vv \in \HH(\ccurl,K)
\\
\curl \vv = \rr_K}
}
\|\vv\|_{0,K}.
\end{equation}
\end{lemma}

\begin{proof}
1) Let us first show that
\begin{equation*}
\|\rr_K\|_{-1,K} \leq \min_{\substack{
\vv \in \HH(\ccurl,K)
\\
\curl \vv = \rr_K
}}
\|\vv\|_{0,K}.
\end{equation*}
Indeed, for every $\vv \in \HH(\ccurl,K)$ such
that $\curl \vv = \rr_K$, we have
\begin{align*}
\|\rr_K\|_{-1,K}
&=
\sup_{\substack{
\pphi \in \HH^1_0(K)
\\
|\pphi|_{1,K} = 1
}}
(\rr_K,\pphi)_{K}
=
\sup_{\substack{
\pphi \in \HH^1_0(K)
\\
|\pphi|_{1,K} = 1
}}
(\curl \vv,\pphi)_{K}
\\
&=
\sup_{\substack{
\pphi \in \HH^1_0(K)
\\
|\pphi|_{1,K} = 1
}}
(\vv,\curl \pphi)_{K}
\leq
\|\vv\|_{0,K}
\left (
\sup_{\substack{
\pphi \in \HH^1_0(K)
\\
|\pphi|_{1,K} = 1
}}
\|\curl \pphi\|_{0,K}
\right )
\leq
\|\vv\|_{0,K},
\end{align*}
since $\|\curl \pphi\|_{0,K} \leq |\pphi|_{1,K}$ for all $\pphi \in \HH^1_0(K)$.
The claim follows by taking the minimum (which exists owing to standard convexity arguments)
over all $\vv \in \HH(\ccurl,K)$ such that $\curl \vv = \rr_K$.
\\
2) Since $\div \rr_{K} = 0$, \cite[Proposition~4.2]{Cost_McInt_Bog_Poinc_10} ensures
the existence of an element $\ww_p \in \NN_p(K)$ such that $\curl \ww_p = \rr_{K}$ and
\begin{equation*}
\|\ww_p\|_{0,K} \lesssim \|\rr_{K}\|_{-1,K}.
\end{equation*}
We can conclude using 1) since
\begin{equation*}
\min_{\substack{
\vv_p \in \NN_p(K)
\\
\curl \vv_p = \rr_{K}
}}
\|\vv_p\|_{0,K}
\leq
\|\ww_p\|_{0,K}
\lesssim
\|\rr_{K}\|_{-1,K}
\leq
\min_{\substack{
\vv \in \HH(\ccurl,K)
\\
\curl \vv = \rr_{K}
}}
\|\vv\|_{0,K}.
\end{equation*}
This proves~\eqref{eq_minimization_zero_face}.
\end{proof}

\subsection{Auxiliary results on the tangential component} \label{sec_aux}

We first establish a density result concerning the space composed of
$\HH(\ccurl,K)$ functions with vanishing tangential trace on $\Gamma_\FF$
in the sense of Definition~\ref{definition_partial_trace}. We consider the subspace
\begin{equation}
\label{eq_definition_HH_Gamma_ccurl}
\HH_{\Gamma_{\FF}}(\ccurl,K)
\eq
\left \{
\vv \in \HH(\ccurl,K) \; | \; \vv|^\ttau_\FF = \boldsymbol 0
\right \},
\end{equation}
equipped with the $\|\cdot\|_{\ccurl,K}$-norm defined above.

\begin{lemma}[Density]
\label{lemma_density}
Let $K \subset \mathbb R^3$ be a non-degenerate tetrahedron and let
$\FF \subseteq \FF_K$ be a nonempty (sub)set of its faces.
The space $\CC^\infty_{\Gamma_\FF}(\overline{K})
\eq
\left \{
\vv \in \CC^\infty(\overline{K}) \; | \;
\vv|_{\Gamma_\FF} = \boldsymbol 0
\right \}$
is dense in $\HH_{\Gamma_\FF}(\ccurl,K)$.
\end{lemma}

\begin{proof}
Recalling~\cite[Remark 3.1]{Fer_Gil_Maxw_BC_97}, if
$\ww \in \HH^{-1/2}(\partial K)$, we can define its restriction
$\ww|_{\Gamma_\FF} \in (\HH^{1/2}_{00}(\Gamma_{\FF}))'$ by setting
\begin{equation} \label{eq_rest}
\langle \ww|_{\Gamma_\FF},\pphi \rangle
\eq
\langle \ww,\widetilde \pphi \rangle_{\partial K}
\quad \forall \pphi \in \HH^{1/2}_{00}(\Gamma_{\FF}),
\end{equation}
where $\widetilde \pphi \in \HH^{1/2}(\partial K)$ denotes the zero-extension of $\pphi$
to $\partial K$. Following~\cite{Fer_Gil_Maxw_BC_97}, we then introduce the space
\begin{equation*}
\VV_{\Gamma_\FF}(K) \eq \left \{
\vv \in \HH(\ccurl,K) \; | \;
(\vv \times \nn)|_{\Gamma_\FF} = \boldsymbol 0
\right \}.
\end{equation*}
Proposition 3.6 of~\cite{Fer_Gil_Maxw_BC_97} states that
$\CC^\infty_{\Gamma_\FF}(\overline{K})$ is dense in $\VV_{\Gamma_\FF}(K)$.
Thus, it remains to show that $\HH_{\Gamma_\FF}(\ccurl,K) \subset \VV_{\Gamma_\FF}(K)$.
Let $\vv \in \HH_{\Gamma_\FF}(\ccurl,K)$. For all $\ttheta \in \HH^{1/2}_{00}(\Gamma_{\FF})$,
we have $\widetilde \ttheta \in \HH^{1/2}(\partial K)$, and there exists
$\pphi \in \HH^1(K)$ such that $\widetilde \ttheta = \pphi|_{\partial K}$.
In addition, since $\widetilde \ttheta|_{\partial K \setminus \Gamma_{\FF}} = \mathbf 0$,
we have $\pphi \in \HH^1_{\FF^{\mathrm c}}(K)$, and in particular
$\pphi \in \HH^1_{\ttau,\FF^{\mathrm c}}(K)$. Then using~\eqref{eq_rest},
integration by parts, and Definition~\ref{definition_partial_trace}, we have
\begin{equation*}
\langle (\vv \times \nn)|_{\Gamma_\FF}, \ttheta \rangle
=
\langle \vv \times \nn, \widetilde \ttheta \rangle_{\partial K}
=
\langle \vv \times \nn, \pphi|_{\partial K} \rangle_{\partial K}
=
(\vv,\curl \pphi)_K - (\curl \vv,\pphi)_K = 0,
\end{equation*}
since $\vv \in \HH_{\Gamma_\FF}(\ccurl,K)$. Hence $(\vv \times \nn)|_{\Gamma_\FF} = \boldsymbol 0$,
and therefore $\vv \in \VV_{\Gamma_\FF}(K)$.
\end{proof}

Since we are going to invoke key lifting results established
in~\cite{Demk_Gop_Sch_ext_II_09,Demk_Gop_Sch_ext_III_12},
we now recall the main notation employed
therein (see~\cite[Section~2]{Demk_Gop_Sch_ext_II_09}). Let
\begin{equation*}
\trc_K^\ttau: \HH(\ccurl,K) \to \HH^{-1/2}(\partial K)
\end{equation*}
be the usual tangential trace operator obtained as in 
Definition~\ref{definition_partial_trace}
with $\FF \eq \FF_K$ and let us equip the image space
\begin{equation*}
\XX^{-1/2}(\partial K) \eq \trc_K^\ttau(\HH(\ccurl,K))
\end{equation*}
with the quotient norm
\begin{equation}
\label{eq_puotient_norm_XX}
\|\ww\|_{\XX^{-1/2}(\partial K)}
\eq
\inf_{\substack{
\vv \in \HH(\ccurl,K)
\\
\trc_K^\ttau(\vv) = \ww
}}
\|\vv\|_{\ccurl,K}.
\end{equation}

For each face $F \in \FF_K$, there exists a Hilbert function space $\XX^{-1/2}(F)$
and a (linear and continuous) ``restriction'' operator
$\RRR_F: \XX^{-1/2}(\partial K) \to \XX^{-1/2}(F)$
that coincides with the usual pointwise restriction for smooth functions.
In particular, we have
\begin{equation} \label{eq:trace_DGS_NN}
\RRR_F(\trc_K^\ttau(\vv_p)) = \ppi^\ttau_F(\vv_p) \qquad \forall \vv_{p} \in \NN_p(K),
\end{equation}
with the tangential trace operator defined in~\eqref{eq:def_pi_tau_F}.
We have thus introduced two notions of ``local traces'' for $\HH(\ccurl,K)$
functions. On the one hand, Definition~\ref{definition_partial_trace} defines
an equality for traces on $\Gamma_{\FF}$ based on integration by parts. On the other hand,
the restriction operators $\RRR_F$ provide another notion of trace on any
face $F \in \FF$. The following result provides a connection
between these two notions.

\begin{lemma}[Trace restriction]
\label{lemma_trace_restriction}
Let $K \subset \mathbb R^3$ be a non-degenerate tetrahedron and let
$\FF \subseteq \FF_K$ be a nonempty (sub)set of its faces.
For all $\rr_\FF \in \NN^\ttau_p(\Gamma_\FF)$ and all $\pphi \in \HH(\ccurl,K)$, if
$\pphi|^\ttau_\FF = \rr_\FF$ according to Definition~\ref{definition_partial_trace}, then
\begin{equation*}
\RRR_F (\trc^\ttau_K (\pphi)) = \rr_F \qquad \forall F \in \FF.
\end{equation*}
\end{lemma}

\begin{proof}
Let $\rr_\FF \in \NN^\ttau_p(\Gamma_\FF)$.
Recalling definition~\eqref{eq_tr_K} of $\NN_p^\ttau(\Gamma_\FF)$ and the
last line of Definition~\ref{definition_partial_trace}, there exists $\vv_p \in \NN_p(K)$
such that $\vv_{p}|^\ttau_\FF = \rr_\FF$. Consider an arbitrary function
$\pphi \in \HH(\ccurl,K)$ satisfying $\pphi|^\ttau_\FF = \rr_\FF$
and set $\widetilde \pphi \eq \pphi - \vv_p \in \HH(\ccurl,K)$. By linearity
we have $\widetilde \pphi|^\ttau_\FF = \mathbf 0$. Using again the fact that
$\vv_p$ is smooth (recall that it is a polynomial), we also have
\begin{equation*}
\RRR_F (\trc_K^\ttau (\vv_p)) = \rr_F \qquad \forall F \in \FF.
\end{equation*}
Thus, by linearity, it remains to show that $\RRR_F(\trc_K^\ttau (\widetilde \pphi)) = \mathbf 0$
for all $F \in \FF$. Recalling~\eqref{eq_definition_HH_Gamma_ccurl}, the identity
$\widetilde \pphi|^\ttau_\FF = \mathbf 0$ means that $\widetilde \pphi \in \HH_{\Gamma_\FF}(\ccurl,K)$.
By Lemma~\ref{lemma_density}, there exists a sequence
$(\widetilde \pphi_m)_{m\in\mathbb N} \subset \CC^\infty_{\Gamma_\FF}(\overline{K})$
that converges to $\widetilde \pphi$ in $\HH_{\Gamma_\FF}(\ccurl,K)$. Now consider a face $F \in \FF$.
Since each function $\widetilde \pphi_m$ is smooth, we easily see that
$\|\RRR_F(\trc^\ttau_K (\widetilde \pphi_m))\|_{\XX^{-1/2}(F)} = 0$.
Then, since the map
$\HH(\ccurl,K) \ni \vv \longmapsto \|\RRR_F(\trc_K^\ttau (\vv))\|_{\XX^{-1/2}(F)} \in \mathbb R$
is continuous, we have
\begin{equation*}
\|\RRR_F(\trc_K^\ttau (\widetilde \pphi))\|_{\XX^{-1/2}(F)} =
\lim_{m \to +\infty}
\|\RRR_F(\trc_K^\ttau (\widetilde \pphi_m))\|_{\XX^{-1/2}(F)} = 0,
\end{equation*}
so that $\RRR_F (\trc_K^\ttau (\widetilde \pphi)) = 0$, which concludes the proof.
\end{proof}

\subsection{Step 2: Minimization without curl constraints}
\label{sec_st_2}

To avoid subtle issues concerning the equivalence of norms, we first establish
the stability result concerning minimization without curl constraints on
the reference tetrahedron $\hK\subset \mathbb R^3$ with vertices $(1,0,0)$,
$(0,1,0)$, $(0,0,1)$, and $(0,0,0)$.

\begin{lemma}[Curl-free minimization, reference tetrahedron]
\label{lemma_lifting_boundary}
Let $\hK\subset \mathbb R^3$ be the reference tetrahedron and let
$\hFF \subseteq \FF_{\hK}$ be a nonempty (sub)set
of its faces. Then, for every polynomial degree $p \geq 0$ and
for all $\hrr_\hFF \in \NN_p^\ttau(\Gamma_\hFF)$ such that
$\scurl_{\hF} (\hrr_{\hF}) = 0$ for all $\hF \in \hFF$, the following holds:
\begin{equation}\label{eq:min_ref_curl_free}
\min_{\substack{
\vv_p \in \NN_p(\hK)
\\
\curl \vv_p = \mathbf 0
\\
\vv_p|^\ttau_{\hFF} = \hrr_\hFF
}}
\|\vv_p\|_{0,\hK}
\lesssim
\min_{\substack{
\vv \in \HH(\ccurl,\hK)
\\
\curl \vv = \mathbf 0
\\
\vv|^\ttau_{\hFF} = \hrr_\hFF
}}
\|\vv\|_{0,\hK}.
\end{equation}
\end{lemma}

\begin{proof}
The proof proceeds in two steps. \\
1) Using a key lifting result that is a direct consequence
of~\cite{Demk_Gop_Sch_ext_II_09,Demk_Gop_Sch_ext_III_12}, let us first establish that
\begin{equation}
\label{eq_stability_dgs}
\min_{\substack{
\vv_p \in \NN_p(\hK)
\\
\curl \vv_p = \mathbf 0
\\
\RRR_{\hF}(\trc_\hK^\ttau (\vv_p)) = \hrr_{\hF} \; \forall \hF \in \hFF
}}
\|\vv_p\|_{0,\hK}
\lesssim
\min_{\substack{
\vv \in \HH(\ccurl,\hK)
\\
\curl \vv = \mathbf 0
\\
\RRR_{\hF}(\trc_\hK^\ttau (\vv)) = \hrr_{\hF} \; \forall \hF \in \hFF
}}
\|\vv\|_{0,\hK}.
\end{equation}
Let us denote respectively by $\vv_p^\star \in \NN_p(\hK)$ and $\vv^\star \in \HH(\ccurl,\hK)$
the discrete and continuous minimizers.
Let us define $\ww^\star \eq \trc_{\hK}^\ttau (\vv^\star) \in \XX^{-1/2}(\partial \hK)$.
Since $\curl \vv^\star = \mathbf 0$, we have $\|\vv^\star\|_{\ccurl,\hK} = \|\vv^\star\|_{0,\hK}$,
and the definition~\eqref{eq_puotient_norm_XX} of the quotient norm
of $\XX^{-1/2}(\partial K)$ implies that
\begin{equation} \label{eq_wwstar}
\|\ww^\star\|_{\XX^{-1/2}(\partial \hK)} \leq \|\vv^\star\|_{0,\hK}.
\end{equation}

Since $\RRR_{\hF}(\trc_\hK^\ttau (\vv^{\star})) = \hrr_{\hF}$, we have
$\RRR_{\hF} (\ww^\star )= \hrr_{\hF}$ for all $\hF \in \hFF$.
We assume that the faces $\FF_{\hK}$ of $\hK$ are numbered as $\hF_1,\dots,\hF_4$
in such a way that the
$n \eq |\hFF|$ first faces are the elements of $\hFF$. We introduce a ``partial lifting''
$\tvv_p \in \NN_p(\hK)$ of $\ww^\star$ using
\cite[Equation (7.1)]{Demk_Gop_Sch_ext_II_09} but taking only
the $n$ first summands. Then, one sees from
\cite[Proof of Theorem 7.2]{Demk_Gop_Sch_ext_II_09} that
\begin{equation} \label{eq_st_DGS}
    \|\tvv_p\|_{\ccurl,\hK} \lesssim \|\ww^\star\|_{\XX^{-1/2}(\partial \hK)}
\end{equation}
and $\RRR_{\hF}(\trc_\hK^\ttau (\tvv_p)) = \hrr_{\hF}$ for all $\hF \in \hFF$.
Thus, relying on~\eqref{eq:trace_DGS_NN} we have $\ppi^\ttau_\hF (\tvv_p) = \hrr_\hF$
for all $\hF \in \hFF$, and we notice that the last line of
Definition~\ref{definition_partial_trace} also equivalently gives
$\tvv_p|^\ttau_\hFF = \hrr_\hFF$.

We must now check that $\curl \tvv_p = \mathbf 0$. This is possible since the
$\HH(\ccurl,\hK)$ and $\HH(\ddiv,\hK)$ trace liftings introduced
in~\cite{Demk_Gop_Sch_ext_II_09,Demk_Gop_Sch_ext_III_12}
commute in appropriate sense. Specifically, recalling that
$\scurl_{\hF} (\hrr_{\hF}) = 0$ for all $\hF \in \hFF$, using the identity
$\scurl_{\hF} (\ppi^\ttau_\hF (\tvv_p)) = \curl \tvv_p \cdot \nn_{\hF}$
valid for all $\hF \in \FF_{\hK}$ (recall that $\nn_{\hF}$ conventionally
points outward $\hK$), see~\eqref{eq_scurl_curl_el}, and with the help of
Theorem 3.1 and Propositions 4.1, 5.1, and 6.1 of~\cite{Demk_Gop_Sch_ext_III_12},
one shows by induction on the summands that $\curl \tvv_p = \mathbf 0$.

Now, since $\tvv_p$ belongs to the discrete minimization set and
using~\eqref{eq_st_DGS} and~\eqref{eq_wwstar}, \eqref{eq_stability_dgs} follows from
\begin{equation*}
\|\vv_p^\star\|_{0,\hK}
\leq
\|\tvv_p\|_{0,\hK}
=
\|\tvv_p\|_{\ccurl,\hK}
\lesssim \|\ww^\star\|_{\XX^{-1/2}(\partial \hK)}
\leq
\|\vv^\star\|_{0,\hK}.
\end{equation*}
2) Let us now establish~\eqref{eq:min_ref_curl_free}.
We first invoke Lemma~\ref{lemma_trace_restriction}. If $\vv \in \HH(\ccurl,\hK)$
satisfies $\vv|^\ttau_\hFF = \hrr_\hFF$, it follows that
$\RRR_{\hF}(\trc_\hK^\ttau (\vv)) = \hrr_{\hF}$ for all $\hF \in \hFF$. As a result, we have
\begin{equation*}
\min_{\substack{
\vv \in \HH(\ccurl,\hK)
\\
\curl \vv = \mathbf 0
\\
\RRR_{\hF}(\trc_\hK^\ttau (\vv)) = \hrr_{\hF} \; \forall \hF \in \hFF
}}
\|\vv\|_{0,\hK}
\leq
\min_{\substack{
\vv \in \HH(\ccurl,\hK)
\\
\curl \vv = \mathbf 0
\\
\vv|^\ttau_\hFF = \hrr_\hFF
}}
\|\vv\|_{0,\hK},
\end{equation*}
the minimization set of the left-hand side being (possibly) larger.
Invoking~\eqref{eq_stability_dgs} then gives
\begin{equation*}
\min_{\substack{
\vv_p \in \NN_p(\hK)
\\
\curl \vv_p = \mathbf 0
\\
\RRR_{\hF}(\trc_\hK^\ttau (\vv_p)) = \hrr_{\hF} \; \forall \hF \in \hFF
}}
\|\vv_p\|_{0,\hK}
\lesssim
\min_{\substack{
\vv \in \HH(\ccurl,\hK)
\\
\curl \vv = \mathbf 0
\\
\vv|^\ttau_\hFF = \hrr_\hFF
}}
\|\vv\|_{0,\hK},
\end{equation*}
and we conclude the proof by observing that
\begin{equation*}
\min_{\substack{
\vv_p \in \NN_p(\hK)
\\
\curl \vv_p = \mathbf 0
\\
\vv_p|^\ttau_\hFF = \hrr_\hFF
}}
\|\vv_p\|_{0,\hK}
=
\min_{\substack{
\vv_p \in \NN_p(\hK)
\\
\curl \vv_p = \mathbf 0
\\
\RRR_{\hF}(\trc_\hK^\ttau (\vv_p)) = \hrr_{\hF} \; \forall \hF \in \hFF
}}
\|\vv_p\|_{0,\hK},
\end{equation*}
the two notions of local trace being equivalent for the discrete functions in $\NN_p(\hK)$.
\end{proof}

To establish the counterpart of Lemma~\ref{lemma_lifting_boundary} in a generic non-degenerate tetrahedron
$K \subset \mathbb R^3$, we are going to invoke the covariant Piola mapping
(see, e.g., \cite[Section~7.2]{Ern_Guermond_FEs_I_2020}).
Consider any invertible affine geometric mapping $\TTT: \mathbb R^3 \to \mathbb R^3$
such that $\Kout = \TTT(\Kin)$. Let $\JJJ_{\TTT}$ be the (constant) Jacobian
matrix of $\TTT$ (we do not require that $\det \JJJ_{\TTT}$ is positive, and in any
case we have $|\det \JJJ_{\TTT}|=|K|/|\hK|$).
The affine mapping $\TTT$ can be identified by specifying
the image of each vertex of $\Kin$. The covariant Piola mapping
$\ppsi^{\mathrm{c}}_{\TTT}:\HH(\ccurl,\Kout) \to \HH(\ccurl,\Kin)$
is defined as follows:
\begin{equation} \label{eq:Piola_c}
\hvv \eq \ppsi^{\mathrm{c}}_{\TTT}(\vv)
=
\left (\JJJ_{\TTT}\right )^{T} \left (\vv \circ \TTT \right ).
\end{equation}
It is well-known that
$\ppsi^{\mathrm{c}}_{\TTT}$ maps bijectively $\NN_p(\Kout)$ to $\NN_p(\Kin)$
for any polynomial degree $p\ge0$. Moreover, for all $\vv\in \HH(\ccurl,\Kout)$, we have
\begin{equation}\label{eq:curl_free}
\curl \vv = \mathbf 0 \; \Longleftrightarrow \curl \hvv = \mathbf 0,
\end{equation}
as well as
the following $\LL^2$-stability properties:
\begin{equation}
\label{eq_stab_piola_L2}
\frac{\rho_{\Kout}}{h_{\Kin}} \|\vv\|_{0,\Kout}
\leq
|\det \JJJ_{\TTT}|^{\frac12}  \|\ppsi_{\TTT}^{\mathrm c}(\vv)\|_{0,\Kin}
\leq
\frac{h_{\Kout}}{\rho_{\Kin}} \|\vv\|_{0,\Kout}.
\end{equation}
Finally the covariant Piola mapping preserves tangential traces. This implies
in particular that for all $F\in\FF_{K}$, setting $\hF \eq \TTT^{-1}(F)$, we have
for all $\vv \in \HH^1(K)$
\begin{equation} \label{equiv_trace_H1}
\ppi^{\ttau}_F(\vv)=\mathbf 0 \; \Longleftrightarrow \; \ppi^{\ttau}_{\hF}(\hvv)=\mathbf 0.
\end{equation}
Finally, for all $\vv\in \HH(\ccurl,\Kout)$, for every nonempty (sub)set
$\FF \subseteq \FF_K$, and for all $\rr_\FF \in \NN_p^{\ttau}(\Gamma_\FF)$, we have
\begin{equation} \label{equiv_trace_Hcurl}
\vv|^\ttau_{\FF}=\rr_\FF \; \Longleftrightarrow \; \hvv|^\ttau_\hFF=\hrr_\hFF,
\end{equation}
where $\hFF\eq \TTT^{-1}(\FF)$ and $\hrr_{\hFF} \in \NN_p^{\ttau}(\Gamma_\hFF)$
is defined such that $\hrr_\hF \eq (\hrr_\hFF)|_{\hF} \eq \ppi^\ttau_{\hF} (\hvv_p)$ for all
$\hF\in\hFF$, where $\hvv_p\eq \ppsi^{\mathrm{c}}_{\TTT}(\vv_p)$ and
$\vv_p$ is any function in $\NN_p^{\ttau}(K)$ such that
$\rr_F \eq (\rr_\FF)|_{F} \eq \ppi^\ttau_{F} (\vv_p)$ for all $F\in\FF$.
The equivalence~\eqref{equiv_trace_Hcurl} is established by using
Definition~\ref{definition_partial_trace}, the properties of
the covariant Piola mapping, and the fact that
$\pphi \in \HH_{\ttau,\FF^{\mathrm{c}}}^1(K)$ if and only if
$\ppsi^{\mathrm{c}}_{\TTT}(\pphi) \in \HH_{\ttau,\hFF^{\mathrm{c}}}^1(\hK)$,
which follows from~\eqref{equiv_trace_H1}.

\begin{lemma}[Curl-free minimization, generic tetrahedron]
\label{lemma_lifting_boundary_K}
Let $K \subset \mathbb R^3$ be a non-degenerate tetrahedron and let
$\FF \subseteq \FF_K$ be a nonempty (sub)set of its faces.
Then, for every polynomial degree $p \geq 0$ and for all
$\rr_\FF \in \NN_p^\ttau(\Gamma_\FF)$ such that
$\scurl_F (\rr_F)=0$ for all $F \in \FF$, the following holds:
\begin{equation}
\label{eq_minimization_element_K_face}
\min_{\substack{
\vv_p \in \NN_p(K)
\\
\curl \vv_p = \mathbf 0
\\
\vv_{p}|^\ttau_\FF = \rr_\FF
}}
\|\vv_p\|_{0,K}
\lesssim 
\min_{\substack{
\vv \in \HH(\ccurl,K)
\\
\curl \vv = \mathbf 0
\\
\vv|^\ttau_\FF = \rr_\FF
}}
\|\vv\|_{0,K}.
\end{equation}
\end{lemma}

\begin{proof}
Consider an invertible affine mapping $\TTT: \hK \to K$ and denote
$\ppsi^{\mathrm{c}}_{\TTT}$ the associated Piola mapping defined in~\eqref{eq:Piola_c}.
Let us set
\begin{alignat*}{2}
{\boldsymbol V}(\hK)&\eq\{\hvv\in \HH(\ccurl,\hK)\ | \ \curl\hvv=\mathbf 0,\ \hvv|^\ttau_\hFF=\hrr_\hFF\}, &\qquad {\boldsymbol V}_p(\hK)&\eq {\boldsymbol V}(\hK) \cap \NN_p(\hK),\\
{\boldsymbol V}(K)&\eq\{\vv\in \HH(\ccurl,K)\ | \ \curl\vv=\mathbf 0,\ \vv|^\ttau_\FF=\rr_\FF\}, &\qquad {\boldsymbol V}_p(K)&\eq {\boldsymbol V}(K) \cap \NN_p(K),
\end{alignat*}
where $\hrr_\hFF$ is defined from $\rr_\FF$ as above.
Owing to~\eqref{eq:curl_free} and~\eqref{equiv_trace_Hcurl}, we infer that
\begin{equation} \label{eq:identities_piola_V}
\ppsi_{\TTT}^{\mathrm c}({\boldsymbol V}(K))={\boldsymbol V}(\hK), \qquad
\ppsi_{\TTT}^{\mathrm c}({\boldsymbol V}_p(K))={\boldsymbol V}_p(\hK).
\end{equation}
One readily checks that $\hrr_\hFF$ satisfies the assumptions of
Lemma~\ref{lemma_lifting_boundary}, so that
\begin{equation*}
\min_{\hvv_p \in{\boldsymbol V}_p(\hK)} \|\hvv_p\|_{0,\hK}
\lesssim
\min_{\hvv \in{\boldsymbol V}(\hK)} \|\hvv\|_{0,\hK}.
\end{equation*}
Invoking the stability properties~\eqref{eq_stab_piola_L2} and the
identities~\eqref{eq:identities_piola_V}, we conclude that
\begin{align*}
\min_{\vv_p \in{\boldsymbol V}_p(K)} \|\vv_p\|_{0,K} &\le
\frac{h_{\Kin}}{\rho_{\Kout}} |\det \JJJ_{\TTT}|^{\frac12}
\min_{\hvv_p \in{\boldsymbol V}_p(\hK)} \|\hvv_p\|_{0,\hK} \\
&\lesssim \frac{h_{\Kin}}{\rho_{\Kout}} |\det \JJJ_{\TTT}|^{\frac12}
\min_{\hvv \in{\boldsymbol V}(\hK)} \|\hvv\|_{0,\hK}
\le \frac{h_{\Kin}}{\rho_{\Kout}} \frac{h_{\Kout}}{\rho_{\Kin}}
\min_{\vv \in{\boldsymbol V}(K)} \|\vv\|_{0,K}.
\end{align*}
This completes the proof.
\end{proof}

\subsection{Step 3: Conclusion of the proof}
\label{sec_st_3}

We are now ready to conclude the proof of Theorem~\ref{theorem_stability_tetrahedra}.
We first apply Lemma~\ref{lemma_lifting_curl} on the tetrahedron $K$ and infer that there exists
$\xxi_p \in \NN_p(K)$ such that $\curl \xxi_p = \rr_K$ and
\begin{equation*}
\|\xxi_p\|_{0,K} \lesssim \min_{\substack{
\vv \in \HH(\ccurl,K)
\\
\curl \vv = \rr_K
}} \|\vv\|_{0,K}.
\end{equation*}
Then, we define $\trr_\FF \in \NN_p^\ttau(\Gamma_\FF)$ by
setting $\trr_F \eq \rr_F - \ppi^\ttau_{F} (\xxi_p)$
for all $F \in \FF$.
Since $\scurl _F(\ppi^\ttau_{F} (\xxi_p)) = \curl \xxi_p \cdot \nn_{F} = \rr_K \cdot \nn_{F}$,
we see that $\scurl_F (\widetilde{\rr}_F) = 0$ for all $F \in \FF$. It follows from
Lemma~\ref{lemma_lifting_boundary_K} that there exists
$\widetilde{\xxi}_p \in \NN_p(K)$ such that $\curl \widetilde{\xxi}_p = \mathbf 0$,
$\widetilde{\xxi}_p|^\ttau_\FF = \widetilde{\rr}_\FF$, and
\begin{equation*}
\|\widetilde{\xxi}_p\|_{0,K}
\lesssim
\min_{\substack{
\vv \in \HH(\ccurl,K)
\\
\curl \vv = \mathbf 0
\\
\vv|^\ttau_\FF = \widetilde{\rr}_\FF
}}
\|\vv\|_{0,K}.
\end{equation*}
We then define $\ww_p \eq \xxi_p + \widetilde{\xxi}_p \in \NN_p(K)$. We observe that
$\ww_p$ belongs to the discrete minimization set of~\eqref{eq_minimization_element_K}.
Thus we have
\begin{equation*}
\min_{\substack{
\vv_p \in \NN_p(K)
\\
\curl \vv_p = \rr_K
\\
\vv_p|^\ttau_\FF = \rr_\FF
}}
\|\vv_p\|_{0,K}
\leq
\|\ww_p\|_{0,K}
\leq
\|\xxi_p\|_{0,K} + \|\widetilde{\xxi}_p\|_{0,K}.
\end{equation*}
Finally we observe that
\begin{equation*}
\|\xxi_p\|_{0,K}
\lesssim
\min_{\substack{
\vv \in \HH(\ccurl,K)
\\
\curl \vv = \rr_K
}}
\|\vv\|_{0,K}
\leq
\min_{\substack{
\vv \in \HH(\ccurl,K)
\\
\curl \vv = \rr_K
\\
\vv|^\tau_\FF = \rr_\FF
}}
\|\vv\|_{0,K},
\end{equation*}
and
\begin{align*}
\|\widetilde{\xxi}_p\|_{0,K}
& \lesssim
\min_{\substack{
\vv \in \HH(\ccurl,K)
\\
\curl \vv = \mathbf 0
\\
\vv|^\tau_\FF = \widetilde{\rr}_\FF
}}
\|\vv\|_{0,K}
=
\min_{\substack{
\widetilde{\vv} \in \HH(\ccurl,K)
\\
\curl \widetilde{\vv} = \curl \xxi_p
\\
\widetilde{\vv}|^\tau_\FF = \trr_\FF + \xxi_p|^\ttau_\FF
}}
\|\widetilde{\vv} - \xxi_p\|_{0,K}
\\
& 
=
\min_{\substack{
\widetilde{\vv} \in \HH(\ccurl,K)
\\
\curl \widetilde{\vv} = \rr_K
\\
\tvv|^\tau_\FF = \rr_\FF
}}
\|\widetilde{\vv} - \xxi_p\|_{0,K}
\leq
\|\xxi_p\|_{0,K} +
\min_{\substack{
\vv \in \HH(\ccurl,K)
\\
\curl \vv = \rr_K
\\
\vv|^\tau_\FF = \rr_\FF
}}
\|\vv\|_{0,K}.
\end{align*}

\bibliographystyle{siam}
\bibliography{biblio}

\end{document}